\pdfoutput=1

\documentclass[12pt]{amsart}

\usepackage[backref]{hyperref}

\usepackage{bm}
\usepackage{amsthm}
\usepackage{graphicx}

\usepackage{amssymb}
\usepackage{faktor}

\usepackage{tikz-cd}

\usepackage{todonotes}
\usepackage[letterpaper,margin=1.0in,top=1.15in,bottom=1.15in]{geometry}

\newcommand{\Or}{\mathrm{O}}

\newcommand{\SO}{\mathrm{SO}}
\newcommand{\Vol}{\mathrm{Vol}}
\newcommand{\C}{\mathbb{C}}
\renewcommand{\H}{\mathbb{H}}
\newcommand{\Q}{\mathbb{Q}}
\newcommand{\Z}{\mathbb{Z}}
\newcommand{\R}{\mathbb{R}}
\newcommand{\PSL}{\mathrm{PSL}}

\newcommand{\SL}{\mathrm{SL}}

\newcommand{\GL}{\mathrm{GL}}

\newtheorem{theorem}{Theorem}[section] 
\newtheorem{theoremx}{Theorem}
\newtheorem{lemma}[theorem]{Lemma}
\newtheorem{corollary}[theorem]{Corollary}
\newtheorem{prop}[theorem]{Proposition}

\newtheorem{rem}[theorem]{Remark}

\newtheorem{ex}[theorem]{Example}

\title[Arithmetic hyperbolic rational homology $3$--spheres]{Infinitely many arithmetic hyperbolic rational homology $3$--spheres that bound geometrically}
\author{L. Ferrari}
\author{A. Kolpakov}
\author{A. W. Reid}

\address{\newline Institut de math\'ematiques
\newline Universit\'e de Neuch\^atel
\newline  Rue Emile--Argand 11
\newline 2000 Neuch\^atel, Switzerland}
\email{leonardocpferrari@gmail.com}
\email{kolpakov.alexander@gmail.com}

\address{\newline Department of Mathematics
\newline Rice University
\newline Houston, TX 77005, USA}

\address{\newline Max-Planck-Insititut f\"ur Mathematik
\newline Vivatsgasse 7 
\newline 53111 Bonn, Germany}
\email{alan.reid@rice.edu, areid@mpim-bonn.mpg.de}

\thanks{L.F. and A.K. were supported by the Swiss National Science Foundation, project no. PP00P2--202667}
\thanks{A.W.R. was supported by the National Science Foundation and the Max--Planck--Institut f\"ur Mathematik}

\begin{document}

\maketitle

\begin{abstract} In this paper we provide the first examples of arithmetic hyperbolic $3$--manifolds that are rational homology spheres and bound geometrically either compact or cusped hyperbolic $4$--manifolds.  
\end{abstract}

\section{Introduction}
\label{intro}
Bordism properties of closed manifolds have been a classical and important topic in topology. To mention but one result, Rokhlin showed that all closed orientable $3$--manifolds bound a compact $4$--manifold. 

In \cite{LR0} the question of  {\em bounding geometrically} was introduced: namely whether a connected closed orientable hyperbolic $n$--manifold $M$ could arise as the totally
geodesic boundary of a compact hyperbolic $(n+1)$--manifold $W$. One could weaken this to merely asking that $M$ bound a finite volume hyperbolic $(n+1)$--manifold with cusps. Another variation of this is to ask if a connected closed orientable {\em flat} $n$--manifold could arise as the cusp cross--section of a finite volume $1$--cusped hyperbolic $(n+1)$--manifold.

In \cite{KRS} another problem was considered: whether a given connected closed orientable hyperbolic $n$--manifold $M$ could {\em embed geodesically}, that is arise as an embedded totally geodesic codimension $1$ submanifold of a hyperbolic $(n+1)$--manifold $W$. 
Although there are obstructions to bounding in certain dimensions, it is now known that in every dimension $n \geq 2$ there are many examples of closed hyperbolic $n$--manifolds which bound geometrically. However, less is known in the case of flat $n$--manifolds.  We refer the reader to \cite{CK, KM, KMT, KRR, LRb, LR0, MZ} for details about constructions of examples, and for description of possible obstructions.  

In particular, in dimension $3$, although many examples of closed orientable hyperbolic $3$--manifolds are known to bound geometrically, all of them up to date have \textit{positive} first Betti number. By virtue of Poincar\'e duality, a $3$-manifold with $\beta_1 = 0$ is a {\em rational homology $3$--sphere}, i.e. $H_q(M,\Q)\cong H_q(\mathbb{S}^3,\Q)$ for all integers $q\geq 0$.

On the other hand, hyperbolic rational homology $3$--spheres 
are abundant: any hyperbolic $(p, q)$ Dehn filling on a hyperbolic knot complement in $\mathbb{S}^3$ will produce an example with integral first homology being $\mathbb{Z}/p\mathbb{Z}$. The smallest closed orientable hyperbolic $3$--manifold, known as the Fomenko--Matveev--Weeks manifold \cite{GMM}, is also an arithmetic rational homology sphere that can be obtained by $\{(5, 1), (5, 2)\}$ Dehn filling of the Whitehead link complement \cite{CFJR}.

Motivated by this, the third author \cite{Reid-lecture} asked whether there are \textit{arithmetic} hyperbolic rational homology $3$--spheres $M$ which bound geometrically. The main results of this paper answer this question. 

\begin{theoremx} \label{main-a}
There are infinitely many hyperbolic rational homology $3$--spheres $M_j$ which bound geometrically a compact hyperbolic $4$--manifold $W_j$. Moreover, there are infinitely many compact hyperbolic $4$--manifolds $W_j$ for which $M_j = \partial W_j$.
\end{theoremx}

\begin{theoremx}\label{main-b}
There are infinitely many hyperbolic rational homology $3$--spheres $X_j$ which bound geometrically a cusped hyperbolic $4$--manifold $Y_j$. Moreover, there are infinitely many cusped hyperbolic $4$--manifolds $Y_j$ for which $X_j=\partial Y_j$.
\end{theoremx}

A common property to both families of manifolds $M_j$ and $X_j$ is that they are all arithmetic of simplest type (see \S \ref{simplest} for details).  That the manifolds are arithmetic of simplest type allows us to use \cite{KRS} to embed these manifolds in closed or cusped arithmetic hyperbolic $4$--manifolds. A further common property of the manifolds $M_j$ and $X_j$ that will be crucial in arranging them to bound (see Lemma \ref{lemma-embed-bound}) is
that $M_j$ and $X_j$ all double cover other rational homology $3$--spheres. The manifolds $M_j$ are commensurable with the arithmetic rational homology $3$--spheres that were constructed in \cite{CD} and \cite{BE}, and the manifolds $X_j$ are commensurable with the group generated by reflections in the right--angled dodecahedron in $\H^3$. 

The key observation that is needed in the proof that $M_j$ and $X_j$ bound geometrically is Lemma \ref{lemma-embed-bound}, which together with Lemma \ref{build_bound}, essentially ``reduces'' our task to group theory. However, the resulting computations rely heavily on Magma \cite{Mag}.

Using more combinatorial and geometric methods via the theory of colourings (see \S \ref{colorQHS}) we can produce some ``sporadic'' examples of 
closed hyperbolic $3$--manifolds $X_j$ which are rational homology spheres and which bound geometrically. In contrast to the former construction, this argument can be made by the ``power of pure thought'' and in a ``computer--free'' way. This was
confirmed by computer as part of a tree-search that found all the possible classes of rational homology $3$--spheres that could be built using colourings of the dodecahedron of lowest rank which bound geometrically.

We end the Introduction by pointing out that both constructions given in the paper
can only produce rational homology $3$--spheres that bound geometrically {\em a non--orientable} hyperbolic $4$--manifold.  It remains open as to whether one can arrange the $4$--manifold to be orientable (as in the constructions of \cite{LR0} for example). This can be traced to Lemma \ref{lemma-embed-bound}, which in our setting cannot be applied to produce an orientable $4$--manifold for which the $3$--manifold bounds geometrically. 

By Lemma \ref{lemma:QHS double-covers}, a rational homology sphere of odd dimension cannot have an orientation--reversing, fixed point free involution, since then it would double--cover a closed non--orientable manifold with trivial reduced rational homology. Such a manifold cannot exist by an Euler characteristic argument: closed manifolds of odd dimension must have $\chi=0$, but a manifold with trivial reduced rational homology has $\chi=1$. In connection with this, we note that \textit{arithmetic} rational homology spheres do not exist in any even dimension $\geq 6$ \cite{Emery}. 

Note that the Fomenko--Matveev--Weeks manifold cannot bound a compact orientable hyperbolic $4$--manifold as the $\eta$--invariant of the former is not an integer \cite{LRb}. Integral $\eta$--invariant would imply vanishing Chern--Simons invariant \cite{MeRu}. The latter can be computed using SnapPy \cite{SnapPy}: for the Fomenko--Matveev--Weeks manifold we obtain $\approx 0.060043$.  

The smallest (arithmetic or not) hyperbolic rational homology $3$--sphere that bounds seems to be unknown as of the time of this writing. 

\medskip

\noindent{\bf Acknowledgements:}~{\normalfont The authors are grateful to the Vinberg Distinguished Lecture Series for the opportunity of fruitful scientific exchange. L.F. and A.K. thank Bruno Martelli for helpful comments on the initial version of the manuscript. A.W.R. thanks Eamonn O'Brien for many hours of Magma tutoring, and thanks the Max--Planck--Institut f\"ur Mathematik, Bonn, for its hospitality during the preparation of this work.}

\section{Arithmetic hyperbolic manifolds}
\label{arithmetic}
For the reader's convenience we recall some facts about arithmetic hyperbolic manifolds. One may find further details in \cite{MR}.

\subsection{Arithmetic hyperbolic $3$--manifolds}
\label{arith3}

Let $M = \H^3/\Gamma$ be a hyperbolic $3$--manifold of finite volume. Then $M$ is called \textit{arithmetic} if the group $\Gamma$ is commensurable with a group $\Gamma_{\mathcal{O}}^1$ as described below.

Let $k$ be a number field with one complex place, $B/k$ a quaternion algebra over $k$, $\mathcal{O}\subset B$ an order. Let $\mathcal{O}^1$ denote the elements of $\mathcal{O}$ of norm $1$, and let $\rho: B\rightarrow M(2,\C)$ be an embedding. Then the group $\Gamma_{\mathcal{O}}^1=\rm{P}\rho(\mathcal{O}^1) \subset \PSL(2,\C)$ is a Kleinian group of finite co-volume.

We say that $\Gamma$ as above is \textit{derived from a quaternion algebra} if $\Gamma < \Gamma_{\mathcal{O}}^1$. 

\subsection{Arithmetic manifolds of simplest type}
\label{simplest}

For the most part, this paper is focused on hyperbolic manifolds of dimension $3$. However, we will need to discuss certain $4$--dimensional hyperbolic
manifolds, namely arithmetic hyperbolic manifolds of {\em simplest type}, whose definition we recall below.

Let $\ell$ be a totally real number field of degree $d$ over $\Q$ equipped with a fixed embedding into
$\R$ which we refer to as the identity embedding. Let $R_\ell$ denote the ring of integers of $\ell$.  Let $V$ be an $(n+1)$--dimensional vector space over $\ell$ equipped with a non--degenerate quadratic form $\mathrm{f}$ defined over $\ell$ which has signature $(n, 1)$ at the identity embedding, and signature $(n+1, 0)$ at the remaining $d-1$ embeddings.

Given this, the quadratic form $\mathrm{f}$ is equivalent over $\R$ to the standard Lorentzian form $\mathrm{J}_n = x_0^2 + x_1^2 + \ldots + x_{n-1}^2 - x_n^2$, and for any non--identity Galois embedding $\sigma:\ell\rightarrow \R$, the quadratic form $\mathrm{f}^\sigma$ (obtained by applying $\sigma$ to each entry of $\mathrm{f}$) is equivalent over $\R$ to $x_0^2 + x_1^2 + \ldots + x_{n-1}^2 + x_n^2$. Such a quadratic form is called {\em admissible}.

Let $F$ be the symmetric matrix associated to $\mathrm{f}$, and let $\Or(\mathrm{f})$ and $\SO(\mathrm{f})$ denote
the linear algebraic groups defined over $k$ as
$\Or(\mathrm{f})=\{X\in\GL(n+1,\C) \,|\, X^tFX=F\} \text{ and } \SO(\mathrm{f})=\{X\in\SL(n+1,\C) \,|\, X^tFX=F\}$. For a subring $L\subset \C$, let the $L$--points of $\Or(\mathrm{f})$, resp.~$\SO(\mathrm{f})$, be denoted by $\Or(\mathrm{f}, L)$, resp. $\SO(\mathrm{f}, L)$.

Note that, given an admissible quadratic form $\mathrm{f}$ defined over $\ell$ of signature $(n,1)$, there exists $T\in \GL(n+1,\R)$ such that $T^{-1}\SO(\mathrm{f}, \R)T = \SO(n, 1)$.  Let ${\rm{Isom}}^+(\H^n)$ denote 
the full group of orientation--preserving isometries of $\H^n$. This can be identified with the group $\SO^+(\mathrm{J}_n, \R) = \SO^+(n,1)$, which is the subgroup of $\SO(n,1)$ preserving the upper half--sheet of the hyperboloid $\{v \in V \,|\, v^T \mathrm{J}_n v = -1\}$.

A subgroup $\Gamma < {\rm{Isom}}^+(\H^n)$ is called {\em arithmetic of simplest type} if $\Gamma$ is commensurable with the image in ${\rm{Isom}}^+(\H^n)$ of $\SO(\mathrm{f}, R_\ell)$ under the conjugation map described above. An arithmetic hyperbolic $n$--manifold $M = {\H}^n/\Gamma$ is called arithmetic of simplest type if $\Gamma$ is of simplest type.
 
The relevance of arithmetic manifolds of simplest type is the following result of \cite{KRS} (see \cite[Proposition 4.1]{KRS} and its proof together with \cite[Section 7]{KRS}). For convenience, in the notation established above, we will say that an orientable arithmetic hyperbolic
$n$--manifold of simplest type $M=\H^n/\Gamma$ is {\em $\ell$--located} if $\Gamma = T\Lambda T^{-1}$ and $\Lambda < \SO(\mathrm{f}, \ell)$.

\begin{theorem}
\label{thm:KRS}
Let $M={\H}^n/\Gamma$ be arithmetic of simplest type which is $\ell$--located. Then $M$ embeds in a hyperbolic $(n+1)$--manifold $N$. If $\ell=\Q$ and $n\geq 3$, then $N$ is a non--compact hyperbolic $(n+1)$--manifold. Moreover, infinitely many distinct commensurability classes of $N$ can be constructed.
\end{theorem}

\begin{rem}\label{rem1}
\normalfont
For $n$ even, all arithmetic hyperbolic $n$--manifolds are of simplest type \cite{VS}.
\end{rem}

\begin{rem}\label{rem2}
\normalfont
If $n=3$, then the class of arithmetic hyperbolic $3$--manifolds of simplest type can be described as precisely those that contain one (and hence infinitely many) totally geodesic surfaces \cite[Chapters 9, 10]{MR}.
In this case the quaternion algebra $B/k$ as in \S \ref{arith3} can be described as $B = A \otimes_\ell k$ where $\ell$ is a totally real number field with $[k:\ell]=2$, and $A$ is a quaternion algebra associated to an immersed totally geodesic surface (see \cite[Theorem 9.5.4]{MR}). The field $\ell$ is the field of definition of the admissible quadratic form $f$ in the description of $M$ as a manifold of simplest type.
\end{rem}

\begin{rem}\label{rem3}
\normalfont
If $M=\H^3/\Gamma$ contains a totally geodesic surface and $\Gamma$ is derived from a quaternion algebra, then it follows from \cite[Chapter 10.2]{MR} that $\Gamma$ is $\ell$--located (where $\ell$ is the maximal totally real subfield of the invariant trace-field of $\Gamma$), and hence satisfies the hypothesis of Theorem \ref{thm:KRS}.
\end{rem}

\section{A criterion for bounding}\label{criterion}

In this section we describe a general construction to arrange for hyperbolic rational $3$--spheres to bound geometrically.

\subsection{Geodesic embeddings and geometric boundaries}\label{sec:fpf-isometries}

We begin by describing a way of promoting geodesic embeddings to bounding geometrically.  

\begin{lemma}\label{lemma-embed-bound}
Let $M$ be an orientable hyperbolic $n$-manifold that has a fixed point free involution $\varphi \in \mathrm{Isom}(M)$. If $M$ embeds geodesically then it also bounds geometrically.
\end{lemma}

\begin{proof}
Let $M$ embed into an orientable manifold $N'$ as a totally geodesic submanifold of codimension $1$. Let us denote by $N$ the manifold obtained by cutting $N'$ along $M$ and taking a connected component. Then either $\partial N = M$ and we are done, or $\partial N = M \sqcup M$, and we can quotient out one copy of $M$ in $\partial N$ by self-identifying it via $\varphi$. Given that $\varphi$ is a fixed point free involution, the resulting metric space $N_\varphi$ will be a hyperbolic manifold with a single boundary component isometric to $M$. Moreover, $N_\varphi$ is orientable or not depending on whether $\varphi$ is orientation-reversing or not. 
\end{proof}

Below we provide two illustrative examples: though none of them is of a hyperbolic manifold bounding another, they give a picture that is easy to visualise.

\begin{ex}[The $2$-torus]
\normalfont
Let $N' = \mathbb{C}/ \Gamma$, where $\Gamma = \langle z \to z+1, z \to z+i \rangle$. Then $N' \cong \mathbb{S}^1 \times \mathbb{S}^1$ is a flat torus. Let $M \cong \mathbb{S}^1$ be embedded into $N'$ as the image of the interval $J = \{ i\cdot t \,|\, t \in [0, 1] \} \subset i\cdot \mathbb{R}$. Then $M$ is a totally geodesic non-separating submanifold of $N'$. 

By cutting $N'$ along $M$, we get the manifold $N\cong \mathbb{S}^1 \times [0,1]$, with $\partial N = M \sqcup M$. Consider then the antipodal map $\varphi$ on $M=\mathbb{S}^1$, which is an orientation-preserving fixed point free involution. The resulting manifold $N_\varphi$ is the M\"obius strip. This is a non-orientable flat $2$-manifold with only one boundary component $M$.
\end{ex}

\begin{ex}[The twisted $I$--bundle]
\normalfont
A higher--dimensional generalisation of the previous example is  the following. Let $\varphi$ be the free (orientation--preserving) involution of a genus $3$ orientable surface $S$ that quotients it down to a genus $2$ surface. Let $N = S \times [0,1]$. When we quotient out $N \times \{1\}$ by $\varphi$, then we obtain $N_\varphi$ that is a twisted $I$--bundle over $S$, and thus cannot be orientable. 
\end{ex}

\begin{rem}[Rokhlin's theorem]
\normalfont
If $M$ is a topological closed $3$--manifold that admits a fixed point free involution $\varphi$, then the proof of Rokhlin's theorem can be reduced to a trivial construction. Namely, taking $W = M \times [0,1]$, so that we can quotient out, say, $M \times \{1\}$ by $\varphi$, and get the desired $N_\varphi$ with $\partial N_\varphi = M \times \{0\}$. 
\end{rem}

\subsection{Towers of rational homology spheres}
\label{get_tower}
The main ideas of the construction build on the works \cite{BE} and \cite{CD}. To state the result that we will make use of, we need to recall some of \cite[Section 6]{CD}.

For an odd prime $p$, a finite $p$--group $S$ is {\em powerful} if $S/S^p$ is Abelian, where $S^p$ is the subgroup of $S$ generated by all $p$--th powers of its elements. When $p=2$, the condition is that $S/S^4$ is Abelian.

A finitely generated group $\Gamma$ is called {\em $p$--powerful} if every finite $p$--group quotient of $\Gamma$ is powerful. 

\begin{prop}\cite{CD}
\label{propCD}
Let $\Gamma$ be a finitely generated group which is $p$--powerful. If $H_1(\Gamma, \Z)$ is finite, then $H_1(H, \Z)$ is finite for any subgroup $H \subset \Gamma$ of $p$--power index.
\end{prop}

Let $G$ be a group: its \textit{$\mathrm{mod}~p$ lower central series} is defined inductively as $\gamma^p_1(G) = G$, with $\gamma^p_{n+1}(G) = \langle (\gamma^p_n(G))^p, [G, \gamma^p_n(G)] \rangle \subseteq \gamma^p_n(G)$ for $n\geq 1$. Then $G$ is {\em residually--$p$} if we have $\bigcap_{n\geq 1} \gamma^p_n(G) = \{1\}$.

\begin{lemma}
\label{build_bound}
Let $M=\H^3/\Gamma$ be an arithmetic hyperbolic rational homology $3$--sphere of simplest type arising from an admissible quadratic form over a totally real field $\ell$ with the following properties:
\begin{enumerate}
\item $\Gamma$ is $\ell$--located;
\item $\Gamma$ is $p$--powerful for some odd prime $p$;
\item $\Gamma$ is residually--$p$;
\item $M$ has a double cover $M' = \H^3/\Delta$ which is a rational homology $3$--sphere, and $\Delta$ is $p$--powerful.
\end{enumerate}
Then there exists a tower of rational homology $3$--spheres $M_j=\H^3/\Delta_j$ which are regular $p$--power coverings of $M'$ that bound geometrically a hyperbolic $4$--manifold $W_j$. In the case when
$\ell=\Q$, the manifold $W_j$ has cusps.\end{lemma}

\begin{proof} By hypotheses, $\Gamma$ is $p$--powerful and residually--$p$, so Proposition \ref{propCD} implies that there exists an infinite tower of $p$--power index normal subgroups $\Gamma_j\lhd \Gamma$ for which all the manifolds
$\H^3/\Gamma_j$ are rational homology $3$-spheres. Set $\Delta_j=\Delta\cap \Gamma_j$.

Since $[\Gamma:\Delta]=2$, it is clear that $\Delta_j \lhd \Gamma_j$. Indeed,
$$\Gamma_j/\Delta_j = \Gamma_j/\Gamma_j\cap\Delta \cong \Delta\Gamma_j/\Delta \subset \Delta\Gamma/\Delta = \Gamma/\Delta \cong \Z/2\Z.$$

Note that $\Delta$ is not a subgroup of $\Gamma_j$ for any $j$ since $[\Gamma:\Delta]=2$, while $[\Gamma:\Gamma_j]=p^k$ for an odd prime $p$ and some integer $k>0$. Thus $[\Gamma_j:\Delta_j] = 2$ for all $j$.
By construction, $\Gamma_j\lhd\Gamma$ with quotient a finite $p$--group, hence $\Delta_j\lhd\Delta$ with quotient a finite $p$--group.  The residually--$p$ condition implies that $\bigcap_{j\geq 1} \Gamma_j = 1$, and so $\bigcap_{j\geq 1} \Delta_j = 1$.

Putting all of this together, since $\Delta$ is $p$--powerful, and each $M_j=\H^3/\Delta_j$ is a $p$--power regular cover of $N$, Proposition \ref{propCD} applies to show that each $M_j$ is a rational homology $3$-sphere. In addition, each $M_j$ double covers the rational homology $3$-sphere $M'_j = \H^3/\Gamma_j$.  

By assumption, $\Gamma$ and thus $\Gamma_j$ and $\Delta_j$ are all arithmetic of simplest type and $\ell$--located.  Hence Theorem \ref{thm:KRS} applies to embed all of the manifolds $M_j$ in an arithmetic hyperbolic $4$--manifold $N_j$.
Note that if $\ell=\Q$ then $N_j$ is necessarily cusped (see \cite{KRS} for example). Regardless, Lemma \ref{lemma-embed-bound} applies to the $M_j$ to promote it from being embedded to bounding geometrically. \end{proof}

\begin{rem} 
\normalfont
As discussed in \cite[Remark 6.5]{CD}, whether a group is $p$--powerful for a given odd prime $p$ can be readily checked, and it reduces to checking whether the maximal finite $p$--group quotient of nilpotency class $2$ is powerful. This can be done in Magma \cite{Mag} via the {\tt{pQuotient}} routine, and a test routine {\tt{IsPowerful}} (see \S \ref{magma} for examples). \end{rem}

\section{Bounding compact hyperbolic $4$--manifolds}
\label{bound_compact}
For the case of bounding a compact manifold, we work with the commensurability class of the group generated by reflections in the faces of the right--angled dodecahedron $\mathcal{D}$ in $\H^3$. This in turn is commensurable with the tetrahedral
group $T = T_4[2,2,3;3,5,2]$ (which is $T_4$ in \cite[Chapter 13.1]{MR}). Indeed, the dodecahedron $\mathcal{D}$ can be split into $120$ copies of its fundamental orthoscheme $T'$ with Schl\"afli symbol $[5,3,4]$ (which is $T_2$ in \cite[Chapter 13.1]{MR}). Two copies of $T'$ glued along an appropriate face produce $T$: one may also think of reflecting $T'$ in one of its faces. We use $T$ instead of $T'$ because it gives rise to a group that is the unit group of a maximal order, and is more convenient for our computations. 

Let $\Gamma$ denote the subgroup of index $2$ in the group generated by reflections in the faces of $T$ consisting of orientation--preserving isometries. A presentation for $\Gamma$ is given by

$$\langle x, y, z \,|\, x^2=y^2=z^3=(yz)^3=(zx)^5=(xy)^2=1 \rangle.$$

The arithmetic information associated to $\Gamma$ is the following. From \cite{MR0}, $\Gamma=\Gamma_{\mathcal{O}}^1$ where $\mathcal{O}$ is a maximal order (unique up to $B^*$--conjugacy) of the quaternion algebra $B/k$,
where $k=\Q(\theta)$, with $\theta^4 - \theta^2 - 1 = 0$, is a degree $4$ complex extension of $\Q$ with two real places, and $B$ is ramified at both of them. Note that the maximal totally real subfield of $k$ is $\ell = \Q(\sqrt{5})$ and, since $\Gamma$ is derived from a quaternion algebra, it follows from Remarks \ref{rem2} and \ref{rem3} that $\Gamma$ is $\ell$--located.

The ring $R_k$ contains two prime ideals of norm $11$. We will use reduction modulo one of these prime ideals, which will be denoted by $\mathcal{P}$, to get an epimorphism $\phi: \Gamma\rightarrow \PSL(2, \mathbb{F}_{11})$, the kernel $\Gamma_1$ of which will provide the initial rational homology $3$--sphere $M = \H^3/\Gamma_1$ to apply Lemma \ref{build_bound}.
From \S \ref{magma_bound_compact}, we see that $H_1(M,\Z)\cong (\Z/2\Z)^7 \oplus (\Z/22\Z)^3$. The Magma routine in \S \ref{magma_bound_compact} establishes that $\Gamma_1$ is $11$--powerful.
By reducing modulo $\mathcal{P}^n$ we get a tower of normal subgroups $\Gamma_j$ of $11$--power index in $\Gamma_1$ with $\bigcap_{j\geq 1} \Gamma_j = \{1\}$. In particular,  $\Gamma_1$ is residually--$11$. 

There are $1023$ subgroups of index $2$ in $\Gamma_1$, and we get Magma to test which of these index $2$ subgroups also have finite abelianisation (there are $363$ of them).  We choose one of these as our subgroup $\Delta$ to apply 
Lemma \ref{build_bound}. Two examples {\tt{M1}} and {\tt{M3}} are taken from this list and Magma certifies that Lemma \ref{build_bound} can be indeed applied.

\begin{rem} \label{rem_vol} 
\normalfont
The smallest volume of one of the rational homology $3$-spheres constructed above equals $4 \cdot |\PSL(2,\mathbb{F}_{11})| \cdot \Vol(T)$ which is approximately $189.4464\ldots$\end{rem}

\section{Bounding cusped hyperbolic $4$-manifolds: building on the examples of \cite{CD} and \cite{BE}}
\label{ex:CD_BE}
We begin by recalling the arithmetic rational homology $3$-spheres of \cite{CD} and \cite{BE}. Thus, 
let $B$ be the quaternion division algebra over $\Q(\sqrt{-2})$ ramified at the prime ideals $\mathcal{P} = \langle 1+\sqrt{-2} \rangle$ and $\overline{\mathcal{P}} = \langle 1-\sqrt{-2} \rangle$ of $\Z[\sqrt{-2}]$ of norm $3$. Let $\mathcal{O}\subset B$ be a maximal order (which is unique up to $B^*$--conjugacy since the type number is $1$), and let 
$\Gamma_{\mathcal{O}}^1$ denote the image in $\PSL(2,\C)$ of the elements of norm $1$. 

In \cite{CD}, Calegari and Dunfield construct a tower of finite index subgroups $\Gamma_j$ in $\Gamma_{\mathcal{O}}^1$ with the following properties: 

\begin{itemize}
\item[(1)] $\Gamma_1=\Gamma_{\mathcal{O}}^1$ and $\Gamma_{j+1} \subset \Gamma_j$;
\item[(2)] $\Gamma_{j+1} \lhd \Gamma_j$ and $\Gamma_j\lhd \Gamma_1$ for all $j$;
\item[(3)]  $\Gamma_j/\Gamma_{j+1}\cong (\Z/3\Z)^2$, resp. $\cong \Z/3\Z$, when $j$ is odd, resp. $j$ is even;
\item[(4)] $\bigcap_{j \geq 1} \Gamma_j = 1$;
\item[(5)] $M_j=\H^3/\Gamma_j$ is a rational homology $3$--sphere for $j\geq 2$.
\end{itemize}

Note that in the construction of the manifolds $M_j$ in \cite{CD}, the fact they were rational homology $3$--spheres was conditional on the Generalised Riemann Hypothesis and part of the Langlands Program, but this was established unconditionally in \cite{BE}.

Another important feature of the commensurability class of $\Gamma_{\mathcal{O}}^1$ is that each group $\Gamma$ commensurable 
with $\Gamma_{\mathcal{O}}^1$ contains arithmetic Fuchsian subgroups, and so if $\Gamma$ is torsion-free, the manifold $\H^3/\Gamma$ contains immersed totally geodesic surfaces (see \cite[Theorem 9.5.4]{MR}). In particular, all the 
manifolds $M_j$, $j\geq 2$, contain immersed totally geodesic surfaces.  Hence, by Remark \ref{rem2} each of the manifolds $M_j$ are of simplest type, and since the totally real subfield of index $2$ is $\Q$, these are simplest type 
for admissible quadratic forms defined over $\Q$. In addition, since each of the groups $\Gamma_j$ are derived from a quaternion algebra, Remark \ref{rem3} applies to
each of the groups $\Gamma_j$ (so they satisfy the hypothesis of Theorem \ref{thm:KRS}), and so the manifolds $M_j$ embed in a cusped hyperbolic $4$--manifold $X_j$.

We will now build a second tower of arithmetic rational homology $3$--spheres $N_j$ with $N_j\rightarrow M_j$ a double cover. The discussion above concerning $M_j$ applies equally well to $N_j$, and so we can deduce that each of the manifolds $N_j$, $j\geq 2$,  embeds in a cusped hyperbolic $4$--manifold. The point about passing to the $N_j$ is that by construction, they admit a free involution and so Lemma \ref{build_bound} will apply to arrange bounding.  Below we provide the necessary details. 

In fact our starting point is the group $\Gamma_2$.  We will make use of a presentation of $\Gamma_2$ computed from that given for $\Gamma_1$ in \S \ref{magma_CD_BE} (as in \cite{CD} and \cite{BE}). As above, we will make use of Magma \cite{Mag} in what follows, and the Magma routine including all the calculations is included in \S \ref{magma_CD_BE}. That $\Gamma_2$ is $3$--powerful is already established in \cite{CD} and \cite{BE}, and from 
the properties of the groups $\Gamma_j$ listed above, we see that $\Gamma_2$ is residually--$3$.

Referring to \S \ref{magma_CD_BE}, we see that $H_1(\Gamma_2,\Z)\cong \Z/6\Z\oplus \Z/6\Z \oplus \Z/36\Z$, and so $\Gamma_2$ has $7$ subgroups of index $2$. 
We will choose one of these subgroups, namely {\tt{L4}} (from the Magma routine in \S \ref{magma_CD_BE}), which we define as $\Delta_2$.  
The construction of our new tower of rational homology $3$--spheres will be completed by applying Lemma \ref{build_bound} once we establish that $\Delta_2$ is $3$--powerful. As before this is certified using Magma \cite{Mag} via the {\tt{pQuotient}} routine, and the routine {\tt{IsPowerful}}.  We refer the reader to \S \ref{magma_CD_BE}.

\begin{rem} 
\normalfont
Using the calculations of \cite{CD} it can be shown that the smallest volume of one of the rational homology $3$--spheres constructed above is approximately $144.5531\ldots$, which is of the same order of magnitude as the example in Remark \ref{rem_vol}.\end{rem}

\section{Examples of $4$--manifolds using Theorem \ref{thm:KRS}}
\label{examples}

We briefly describe how to implement Theorem \ref{thm:KRS} to provide infinitely many commensurability classes of closed and cusped hyperbolic $4$--manifolds $Y_j$ and $W_j$ for which $X_j$ and $N_j$ embeds, thereby allowing us to conclude the proof of Theorems \ref{main-a} and \ref{main-b}. To do this, we need to construct an admissible
quadratic form over a totally real field.

\medskip

\paragraph{\bf Closed case.} As follows from \cite{Bug} the group $\Gamma$ is a subgroup in the group $\mathrm{O}(\mathrm{f}, R_\ell)$ of the admissible quadratic form $\mathrm{f} = x^2_1 + x^2_2 + x^2_3 - \frac{1 + \sqrt{5}}{2} x^2_4$ over the field $\ell = \mathbb{Q}(\sqrt{5})$ with the ring of integers $R_\ell = \Z[{\frac{1+\sqrt{5}}{2}}]$. Let $\mathrm{q} = x^2_0 + f$. The separability arguments from \cite{KRS} can be adopted so that we produce a tower of manifold coverings $N'_i \rightarrow \mathbb{H}^4 / \mathrm{SO}(\mathrm{q}, R_\ell)$, for $i = 1, 2, \ldots$, of ever increasing degrees (and thus having different volumes), such that $M_j = \mathbb{H}^3 / \Gamma_j$ embeds in each $N'_i$. By applying Lemma~\ref{lemma-embed-bound} to each $N'_i$ we get an an infinite sequence $W'_i$ with $\partial W'_i = M_j$. Thus, in Theorem~\ref{main-a}, we can set $W_j = W'_i$, for any $i = 1, 2, \ldots$

\medskip

\paragraph{\bf Cusped case.} The quaternion algebra $B/\Q(\sqrt{-2})$ used by \cite{CD} can be described via a Hilbert symbol as $\biggl({\frac{-1,3}{\Q}}\biggr)\otimes_\Q \Q(\sqrt{-2})$.  Using \cite{MR0} or \cite[Chapter 10.2]{MR}, an admissible
quadratic form is $\mathrm{f} = x^2_1+6x^2_2+6x^2_3-2x^2_4$. Now let $\mathrm{q} =  x^2_0 + \mathrm{f}$ and apply the above argument to get infinitely many rational homology $3$--spheres $X_j$ embedding each into infinitely many manifolds $N'_i$, and thus each bounding infinitely many $W'_i$'s. The only difference being that $W'_i$ are each cusped. Then we can set $Y_j = W'_i$ for any $i=1,2,\ldots$ 

\section{Colourings and rational homology $3$--spheres}
\label{colorQHS}

In this section we provide a more concrete construction of some "sporadic" rational homology $3$--spheres that bound geometrically.  These will be built from the all right dodecahedron in $\H^3$, and will be commensurable with the examples in \S \ref{bound_compact}. The details are given in the subsections below.

\subsection{Colourings of right--angled polyhedra}\label{sec:colourings-prelim}

A finite-volume polytope $\mathcal{P} \subset \mathbb{X}^n$ (for $\mathbb{X}^n = \mathbb{S}^n, \mathbb{E}^n, \mathbb{H}^n$ being spherical, Euclidean and hyperbolic $n$--dimensional space, respectively, see \cite[Chapters 1--3]{Ratcliffe}) is called {\itshape right--angled} if any two codimension $1$ faces (or facets, for short) are either intersecting at a right angle or disjoint. It is known that compact hyperbolic right-angled polytopes cannot exist if $n > 4$ \cite{PV}. The only compact right-angled spherical and Euclidean polytopes are the $n$--simplex and the $n$--parallelotope, respectively. A sufficient condition for an abstract $3$--polytope to be realisable as right--angled hyperbolic one is given in \cite[Theorem 2.4]{Vesnin}. There is no such classification for right-angled $n$--polytopes with $n\geq 4$. We refer the reader to \cite{Dufour, Kolpakov, PV, Vesnin} for more information on right--angled polytopes.

One of the important properties of hyperbolic right-angled polytopes is that their so-called colourings provide a rich class of hyperbolic manifolds. By inspecting the combinatorics of a colouring, one may obtain important topological and geometric information about the associated manifold.

Let $\mathcal{P} \subset \mathbb{X}^n$ be a compact, right-angled polytope with the set of facets $\mathcal{F}$. A \textit{colouring} of $\mathcal{P}$ is a map $\lambda: \mathcal{F} \rightarrow W$, where $W$ an $\mathbb{Z}/2 \mathbb{Z}$--vector space. The map $\lambda$ is called \textit{proper} if, for every vertex $v=F_1\cap \ldots \cap F_n$, the vectors $\lambda(F_1), \ldots, \lambda(F_n)$ are linearly independent.

If the polytope $\mathcal{P}$ or the vector space $W$ are clear from the context, then we will omit them and simply refer to $\lambda$ as a colouring. The \textit{rank} of $\lambda$ is the $\mathbb{Z}/2\mathbb{Z}$--dimension of $\mathrm{im}\, \lambda$. We will always assume that colourings are surjective, in the sense that the image of the map $\lambda$ is a generating set of vectors for $W$.

A colouring of a right-angled $n$--polytope $\mathcal{P}$ naturally defines a homomorphism, which we still denote by $\lambda$ without much ambiguity, from the associated right-angled Coxeter group~$\Gamma(\mathcal{P})$, that is generated by reflections in all the facets of $\mathcal{P}$, into $W$ with its natural group structure. Being a Coxeter polytope, $\mathcal{P}$ has a natural orbifold structure as the quotient $\mathbb{X}^n /_{\Gamma(\mathcal{P)}}$.

\begin{prop}[\cite{DJ91}, Proposition 1.7]\label{DJ}
If the colouring $\lambda$ is proper, then $\ker \lambda < \Gamma(\mathcal{P})$ is torsion-free, and $\mathcal{M}_\lambda = \mathbb{X}^n /_{\ker \lambda}$ is a closed manifold.
\end{prop}

Notice that if $\mathcal{P} \subset \mathbb{X}^n$ is a compact right-angled polytope then $\mathcal{P}$ is necessarily simple and its dual $K=(\partial \mathcal{P})^*$ is a simplicial complex.

We say that a $(\mathbb{Z}/2\mathbb{Z})^k$-colouring $\lambda$ is \textit{orientable} if the orbifold $\mathcal{M}_\lambda$ is orientable. We have the following criterion for orientability.

\begin{prop}[\cite{KMT}, Lemma 2.4]\label{prop:orientability}
The orbifold $\mathcal{M}_\lambda$ is orientable if and only if $\lambda$ is equivalent to a colouring that assigns to each facet a colour in $W \cong (\mathbb{Z}/2\mathbb{Z})^k$ with an odd number of entries $1$.
\end{prop}

Given a right-angled polytope $\mathcal{P} \subset \mathbb{X}^n$ with a $(\mathbb{Z}/2\mathbb{Z})^k$--colouring $\lambda$, let us enumerate the facets $\mathcal{F}$ of $\mathcal{P}$ in some order. Then we can assume that $\mathcal{F} = \{ 1, 2, \ldots, m \}$. Let $\Lambda$ be \textit{the defining matrix} of $\lambda$ that consists of the column vectors $\lambda(1), \ldots, \lambda(m)$ exactly in this order. Hence $\Lambda$ is a matrix with $k$ rows and $m$ columns. More precisely, $\Lambda$ represents the abelianisation of $\lambda$, i.e. the former is a map such that $\Lambda \circ \mathrm{ab} = \lambda$, where $\mathrm{ab}:\Gamma \to (\mathbb{Z}/2\mathbb{Z})^m$ is the abelianisation map that takes $r_i$, the reflection of the facet $i$, to $e_i$.

Let $\mathrm{Row}(\Lambda)$ denote the row space of $\Lambda$, while for a vector $\omega \in \mathrm{Row}(\Lambda)$ let $K_\omega$ be the simplicial subcomplex of the complex $K=K_\mathcal{P}$ spanned by the vertices $i$, also labelled by the elements of $\{1, 2, \ldots, m\}$, such that the $i$--th entry of $\omega$ equals $1$.

Then the rational cohomology of $\mathcal{M}_{\lambda}$ can be computed via the following formula, cf. \cite[Theorem 1.1]{CP2}.

\begin{equation}\label{eq:cohomology}
H^p(\mathcal{M}_{\lambda},\mathbb{Q})\cong \underset{\omega \in \mathrm{Row}(\Lambda)}{\bigoplus}\widetilde{H}^{p-1}(K_{\omega},\mathbb{Q}). 
\end{equation}

Moreover, the cup product structure is given by the maps \cite[Main Theorem]{CP2}:

\begin{equation}\label{eq:cup products}
\widetilde{H}^{p-1}(K_{\omega_1},\mathbb{Q}) \otimes \widetilde{H}^{q-1}(K_{\omega_2},\mathbb{Q}) \mapsto \widetilde{H}^{p+q-1}(K_{\omega_1 + \omega_2},\mathbb{Q}). 
\end{equation}

\subsection{Colouring extensions}\label{sec:extensions}

Let $\lambda:\mathcal{F}\to (\mathbb{Z}/2\mathbb{Z})^k$ be any colouring. A (surjective) colouring $\mu:\mathcal{F}\to (\mathbb{Z}/2\mathbb{Z})^{k+1}$ is called \textit{an extension of} $\lambda$ if there is a linear projection $p:(\mathbb{Z}/2\mathbb{Z})^{k+1}\to (\mathbb{Z}/2\mathbb{Z})^k$ such that $\lambda=p\circ\mu$.

\begin{prop}\label{prop:double-cover colouring}
Let $\lambda:\mathcal{F}\to (\mathbb{Z}/2\mathbb{Z})^k$ be any colouring and $\mu:\mathcal{F}\to (\mathbb{Z}/2\mathbb{Z})^{k+1}$ its extension. Then $\mathcal{M}_\mu$ double-covers $\mathcal{M}_\lambda$. Moreover, if $\lambda$ is proper or orientable, so is $\mu$.
\end{prop}

\begin{proof}
Let $\Gamma=\Gamma(\mathcal{P})$ be the reflection group associated with $\mathcal{P}$. Let $\lambda:\Gamma \to (\mathbb{Z}/2\mathbb{Z})^k$ and $\mu:\Gamma \to (\mathbb{Z}/2\mathbb{Z})^{k+1}$ be the homomorphisms induced by $\lambda$ and $\mu$, respectively. By definition, we have that $\lambda = p \circ \mu$, and it follows that $\mathrm{ker} \, \lambda=\mathrm{ker} \, (p \circ \mu)=\mu^{-1}(\mathrm{ker}\, p)$. Moreover, $\mathrm{Im}\,p\cong (\mathbb{Z}/2\mathbb{Z})^k$ and $|\mathrm{ker} \, p|=[(\mathbb{Z}/2\mathbb{Z})^{k+1}:\mathrm{Im}\,p]=2$. Thus $\mathrm{ker} \, p=\{0,v_0\}$ for some $v_0\in (\mathbb{Z}/2\mathbb{Z})^{k+1}$, $v_0 \neq 0$.

Since $\mu$ is surjective, there exists $u_0\in \mu^{-1}(v_0)\neq \emptyset$. Since $\mu$ is a homomorphism, $\mu^{-1}(v_0)=u_0+ \mathrm{ker}\, \mu $. Then $\mathrm{ker} \, \lambda = \mathrm{ker} \, \mu \sqcup (u_0+\mathrm{ker}\,\mu)$, and thus $\mathrm{ker}\,\mu \triangleleft_{2} \mathrm{ker}\, \lambda$. Hence $\mathcal{M}_\mu$ is a double cover of $\mathcal{M}_\lambda$. 

Finally, assume that $\{\lambda(F_1), \ldots, \lambda(F_s)\} \subset (\mathbb{Z}/2\mathbb{Z})^k$ is a set of linearly independent colours. By using the fact that $\lambda = p \circ \mu$, we easily obtain that  $\mu(F_1), \ldots, \mu(F_s)$ are linearly independent. Hence, if $\lambda$ is proper then $\mu$ is proper too. Also, if $\mathcal{M}_\mu$ double-covers $\mathcal{M}_\lambda$ and the latter is orientable, so is $\mathcal{M}_\mu$.
\end{proof}

One direct application of Equation~\eqref{eq:cohomology} to extensions of colourings is the following.

\begin{prop}\label{prop:double-cover properties}
Let $\lambda:\Gamma \to (\mathbb{Z}/2\mathbb{Z})^k$ be a colouring and $\mu$ its extension. Let also $\Lambda$ and $\mathrm{M}$ be their respective defining matrices. Then, up to equivalence, $\mathrm{M}$ is the matrix obtained from $\Lambda$ by adding an extra row vector $v\in (\mathbb{Z}/2\mathbb{Z})^m =\mathrm{ab}(\Gamma)$, such that $v\notin \mathrm{Row}(\Lambda)$. Moreover, if $\lambda$ is orientable, so is $\mu$. Finally, for all $p\ge 0$,
\begin{equation*}
    H^p(\mathcal{M}_\mu,\mathbb{Q})=H^p(\mathcal{M}_\lambda,\mathbb{Q})\oplus \left( \underset{\omega \in \mathrm{Row}(\Lambda)}{\bigoplus}\widetilde{H}^{p-1}(K_{\omega+v},\mathbb{Q}) \right). 
\end{equation*}
\end{prop}

\begin{proof}
Up to isomorphism, we may assume the projection $p : (\mathbb{Z}/2\mathbb{Z})^{k+1} \to (\mathbb{Z}/2\mathbb{Z})^k$ is just the canonical projection onto the first $k$ coordinates. Then, since $p\circ \mathrm{M} = \Lambda$, it is clear that $\mathrm{M}$ is the matrix $\Lambda$ with another row $v\in (\mathbb{Z}/2\mathbb{Z})^m$ added. Moreover, $\mu$ is surjective if and only if $\mathrm{M}$ is surjective, and the latter holds if and only if $v\notin \mathrm{Row}(\Lambda)$. The colouring extensions can be seen in red in the diagram below:

\vspace*{0.55in}
$$
\begin{tikzcd}
\Gamma \arrow[overlay, r, red, "\mu"] \arrow[overlay, d, "ab"] \arrow[overlay, dr, out=90, in=0, looseness=2.5, "\lambda"] 
& (\mathbb{Z}/2\mathbb{Z})^{k+1} \arrow[overlay, d, "p"] \\
(\mathbb{Z}/2\mathbb{Z})^m \arrow[overlay, r, "\Lambda"] \arrow[overlay, ur, red, crossing over, "M"]
& (\mathbb{Z}/2\mathbb{Z})^k
\end{tikzcd}
$$

Clearly, $\mathrm{Row}(\mathrm{M})=\mathrm{Row}(\Lambda)\sqcup\big(v+\mathrm{Row}(\Lambda)\big)$. We conclude by applying Equation~\eqref{eq:cohomology}.
\end{proof}

Conversely, there is a criterion to tell whether a given colouring $\mu$ is an extension of some other colouring $\lambda$.

\begin{prop}\label{prop:stabilizer condition}
Let $\mu:\Gamma(\mathcal{P})\to W$ be a proper colouring, and let $W_p = \mu\big(\mathrm{Stab}_\Gamma(p)\big)$ for any vertex $p$ of $P$. Then $\mu$ is an extension of some proper colouring if and only if $\bigcup_p W_p \subsetneq W$.
\end{prop}

\begin{proof}
Assume that there is a projection $p : W \cong (\mathbb{Z}/2\mathbb{Z})^k \to (\mathbb{Z}/2\mathbb{Z})^{k-1}$ such that $p\circ \mu$ is a proper colouring. Then, for any codimension $s$ face $f = F_1\cap \ldots \cap F_s$ of $\mathcal{P}$ we have $(p\circ\mu)(F_1) + \ldots + (p\circ\mu)(F_s)\neq 0$. This means that $\mu(F_1) + \ldots + \mu(F_s)\notin \mathrm{ker}\, p$. As in the proof of Proposition~\ref{prop:double-cover colouring}, we have that $\mathrm{ker}\, p=\{0,v_0\}$ for some $v_0\in W$ and, in particular, $\mu(F_1) + \ldots + \mu(F_s)\neq v_0$ for any face $f=F_1\cap \ldots \cap F_s$. It follows that $v_0\notin W_f$ for any such face $f$ and, in particular, for any vertex $q\in \mathcal{P}$.

Conversely, assume there is a vector $v_0\in W\setminus \bigcup_{q} W_q$. Then $\mu(F_1)+ \ldots +\mu(F_s) \neq v_0$ for any codimension $s$ face $f = F_1\cap \ldots \cap F_s$ of $\mathcal{P}$. Let $W\cong (\mathbb{Z}/2\mathbb{Z})^k$ and $p : (\mathbb{Z}/2\mathbb{Z})^k \to (\mathbb{Z}/2\mathbb{Z})^{k-1}$ be the projection along $v_0$. Since $\mu$ is proper, we also have that $\mu(F_1)+ \ldots +\mu(F_s) \neq 0$ for any face $f = F_1\cap \ldots \cap F_s$, that is, $\mu(F_1)+ \ldots +\mu(F_s) \notin \{0,v_0\}=\mathrm{ker} \, p$. Let us then set $\lambda = p\circ \mu$. This is a proper colouring since $\lambda(F_1)+ \ldots +\lambda(F_s)\notin p(\mathrm{ker}\, p)=\{0\}$ for all faces $f = F_1\cap \ldots \cap F_s$. By definition, $\mu$ is an extension of $\lambda$.
\end{proof}

\begin{ex}[The Hantzsche--Wendt colouring]
\normalfont
Let $\lambda$ be the colouring of the $3$--cube defined in \cite[p.~8]{FKS} such that $\mathcal{M}_\lambda$ is the Hantzsche--Wendt manifold \cite{HW}. In particular, $\mathrm{rank}\,\lambda = 4$. However, we have that $\bigcup_{p} W_q = W$, and it follows from Proposition \ref{prop:stabilizer condition} that $\lambda$ is not an extension of any colouring.
\end{ex}

\subsection{Rational homology $3$--spheres}\label{sec:3-qhs}

We say that a $CW$--complex is a \textit{rational homology point} if all its reduced  $\mathbb{Q}$--homology groups are trivial.

Let $\epsilon = (1, \ldots, 1) \in (\mathbb{Z}/2\mathbb{Z})^m$. By Proposition~\ref{prop:orientability}, we have that $\epsilon \in \mathrm{row} \, \Lambda$ for every orientable $\lambda$, since it's given by the sum of rows of $\Lambda$. By applying Equation~\eqref{eq:cohomology}, we have the following.

\begin{lemma}\label{lemma:subcomplexes are HP}
An orientable $\mathcal{M}_{\lambda}$ is a rational homology sphere if and only if for all $\omega \in \mathrm{Row}(\Lambda) \setminus \{0, \varepsilon\}$, $K_\omega$ is a rational homology point.
\end{lemma}

\begin{proof}
The only non-trivial cohomology groups of $\mathcal{M}_{\lambda}$ are $H_n(\mathcal{M}_{\lambda},\mathbb{Q})\cong H^0(\mathcal{M}_{\lambda},\mathbb{Q})\cong \widetilde{H}^{-1}(K_0,\mathbb{Q})\cong \mathbb{Q}$ and $H_0(\mathcal{M}_{\lambda},\mathbb{Q})\cong H^n(\mathcal{M}_{\lambda},\mathbb{Q})\cong \widetilde{H}^{n-1}(K_\varepsilon,\mathbb{Q})\cong \mathbb{Q}$. Therefore, every other simplicial subcomplex $K_\omega$ must have trivial reduced homology groups.
\end{proof}

By applying Equation~\eqref{eq:cup products}, we get a useful consequence. 

\begin{lemma}\label{lemma:HP condition}
Let $\mathcal{M}_{\lambda}$ be an orientable $n$--manifold and $\omega \in \mathrm{Row}(\Lambda)\setminus\{0,\epsilon\}$. Then $K_\omega$ is a rational homology point if and only if $K_{\epsilon-\omega}$ is so.
\end{lemma}

\begin{proof}
Assume that $K_\omega$ is not a rational homology point. Then $\tilde{H}^*(K_\omega,\mathbb{Q})$ is non--trivial. Let $0\neq \alpha \in \widetilde{H}^i(K_{\omega},\mathbb{Q})$ for some $i\in \{0, \ldots, n-2\}$. By Equation~\eqref{eq:cohomology}, $\alpha \in H^{i+1}(\mathcal{M}_{\lambda},\mathbb{Q})$. By \cite[Corollary 3.39]{Hatcher}, there exists $\beta \in H^{n-i-1}(\mathcal{M}_{\lambda},\mathbb{Q})$ such that $\alpha\smile \beta$ is the generator of $H^n(\mathcal{M}_{\lambda},\mathbb{Q})$. Then, by Equation~\eqref{eq:cohomology}, $\alpha \smile \beta$ is the generator of $\widetilde{H}^{n-1}(K_\epsilon, \mathbb{Q})$, since $K_\epsilon$ is homotopically $\mathbb{S}^{n-1}$. Finally, by Equations~\eqref{eq:cohomology}--\eqref{eq:cup products}, we obtain that $0\neq \beta \in \widetilde{H}^{n-i-2}(K_{\epsilon-\omega},\mathbb{Q})$, since otherwise the product $\alpha \smile \beta$ would not belong to $\widetilde{H}^{n-1}(K_\epsilon, \mathbb{Q})$. 
\end{proof}

Thus, we can improve Lemma \ref{lemma:subcomplexes are HP} algorithmically by checking only the connectivity of some graphs. 

\begin{corollary}\label{cor:algorithm}
An orientable $3$--manifold $\mathcal{M}_{\lambda}$ is a rational homology sphere if and only if for all $\omega \in \mathrm{Row}(\Lambda) \setminus \{0, \varepsilon\}$ the $1$--skeleton of $K_\omega$ is connected.
\end{corollary}

\begin{proof}
If one proper subcomplex $K_\omega$ has a non-trivial cycle then, by Lemma \ref{lemma:HP condition}, the complementary complex $K_{\varepsilon-\omega}$ will be disconnected. It suffices therefore to check if all proper, non-empty subcomplexes $K_\omega$ are connected. Clearly the connectivity of $K_\omega$ depends only on the connectivity of its $1$--skeleton.
\end{proof}

In the case of double covers, the transfer homomorphisms \cite[Section 3.G]{Hatcher} can be easily used in order to obtain the following statement:

\begin{lemma}\label{lemma:QHS double-covers}
Let $Y$ be a closed manifold that is a double cover of another manifold $X$. If $Y$ is a rational homology sphere, then $X$ is either a rational homology sphere or a rational homology point.
\end{lemma}

Thus, if we want to obtain a colouring $\mu$ producing a $3$--dimensional rational homology sphere, such that $\mu$ is an extension of a proper colouring $\lambda$, then we need that the starting colouring $\lambda$ also produce a rational homology sphere. In this regard, we can use the following algorithm.

\begin{lemma}\label{lemma:QHS extension algorithm}
Let $\lambda$ be a proper colouring such that the $3$--manifold $\mathcal{M}_\lambda$ is a rational homology sphere. Let $\mu$ be any extension of $\lambda$, obtained by adding to $\Lambda$ a row vector $v \notin \mathrm{row} \, \Lambda$. Then $\mathcal{M}_\mu$ is a rational homology sphere if and only if for every pair $\{\omega,\epsilon-\omega\} \subset \mathrm{Row}(\Lambda)$, we have that $K_{\omega+v}$ is connected and has only trivial homology $1$--cycles.
\end{lemma}

\begin{proof}
By Proposition \ref{prop:double-cover properties}, we have that $\mathcal{M}_\mu$ is a rational homology sphere if and only if $K_{\omega+v}$ is homologically trivial for every $\omega \in \mathrm{Row}(\Lambda)$. By Lemma \ref{lemma:HP condition}, $K_{\omega+v}$ is homologically trivial if and only if $K_{\epsilon-(\omega+v)}$ is so. Since $\epsilon-(\omega+v)=(\epsilon-\omega)+v$ and $\epsilon \in \mathrm{Row}(\Lambda)$, we have that also $\epsilon-\omega \in \mathrm{Row}(\Lambda)$. Therefore, it is enough to check whether $K_{\omega+v}$ is homologically trivial for each pair $\{\omega,\epsilon-\omega\} \subset \mathrm{Row}(\Lambda)$. Since $K$ is homeomorphic to $\mathbb{S}^2$ and $K_{\omega + v}$ is a proper subcomplex of $K$, then  $K_{\omega + v}$ is homologically trivial if and only if it is connected and has only trivial homology $1$--cycles.
\end{proof}

\subsection{A rational homology sphere from colouring that bounds geometrically}

\begin{prop}\label{prop:arithmetic colourings that bound}
Let $\lambda:\Gamma(\mathcal{P})\to W$ be a proper colouring of the hyperbolic, compact, right--angled $3$--polytope $\mathcal{P}$ with arithmetic reflection group $\Gamma = \Gamma(\mathcal{P})$. If $\bigcup_q W_q \subsetneq W$, then $M_\lambda$ bounds geometrically. Equivalently, any extension of a proper colouring of $\mathcal{P}$ bounds geometrically.
\end{prop}

\begin{proof}
By \cite{Tits}, and $\Gamma$ is of simplest type, and by \cite[Theorem~5]{Vinberg} we have that $\Gamma$ is also $k$--located. Then $\mathcal{M}_\lambda = \mathbb{H}^3/\Gamma_\lambda$ is an arithmetic manifold with $k$--located  $\Gamma_\lambda$, for any proper colouring of $\mathcal{P}$. Thus, $\mathcal{M}_\lambda$ embeds geodesically by Theorem \ref{thm:KRS}.

If $\bigcup_q W_q \subsetneq W$ then, by Proposition \ref{prop:stabilizer condition}, $\lambda$ is an extension of some colouring $\mu$ and, by Proposition~\ref{prop:double-cover colouring}, we have that $\mathcal{M}_\lambda$ double-covers $\mathcal{M}_\mu$. Then $\mathcal{M}_\lambda$ has a fixed point free involution and therefore bounds geometrically by Lemma \ref{lemma-embed-bound}.
\end{proof}

\begin{theorem}\label{thm:arithmetic-QHS-that-bounds}
There is a colouring $\mu$ of the right-angled dodecahedron such that $\mathcal{M}_\mu$ is an arithmetic hyperbolic rational homology $3$--sphere that bounds geometrically.
\end{theorem}

\begin{proof}
Let $\mathcal{D}$ be the right-angled dodecahedron and take the only orientable small cover $\lambda$ of $\mathcal{D}$ given in \cite[p. 6]{GS}. Thus
\begin{equation*}
\Lambda=\begin{pmatrix}
1 & 0 & 0 & 0 & 0 & 1 & 1 & 1 & 1 & 1 & 0 & 0 \\
0 & 1 & 0 & 0 & 1 & 1 & 0 & 1 & 1 & 0 & 0 & 1 \\
0 & 0 & 1 & 1 & 0 & 1 & 0 & 1 & 1 & 0 & 1 & 0 \\
\end{pmatrix},
\end{equation*}
where the labeling of the faces of $\mathcal{D}$ is given in Figure \ref{im:dodecahedron}. By Equation~\eqref{eq:cohomology}, $\mathcal{M}_\lambda$ is a rational homology sphere. If we find an extension $\mu$ of $\lambda$ such that $\mathcal{M}_\mu$ is also a rational homology sphere, then we are done by Proposition \ref{prop:arithmetic colourings that bound}, since $\Gamma(\mathcal{D})$ is arithmetic by \cite[Lemma 3.8]{Vesnin}.

\begin{figure}[h]
    \centering
    \includegraphics[scale=0.42]{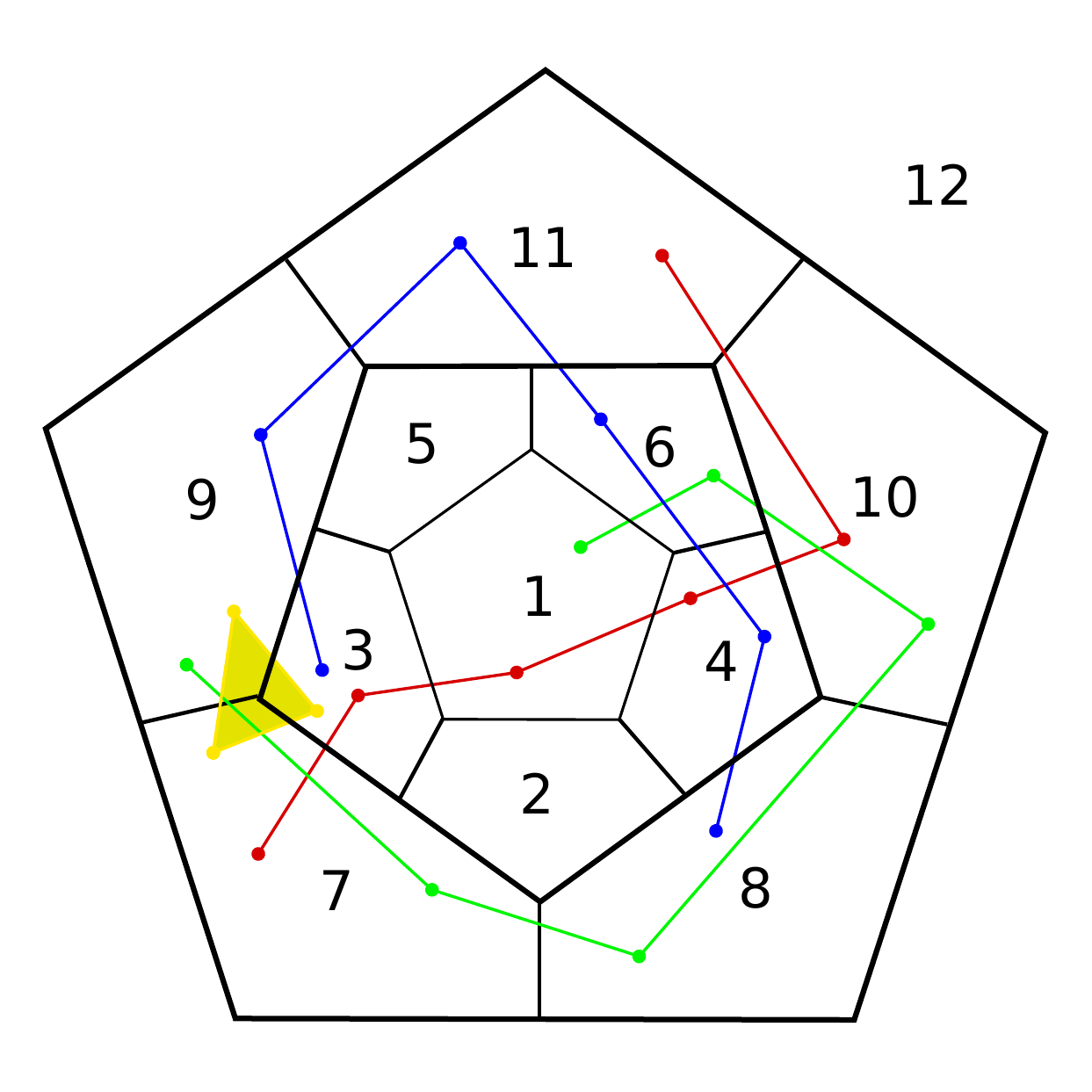}
    \caption{The dodecahedron used in the proof of Theorem~\ref{thm:arithmetic-QHS-that-bounds} with its face labelling. The red, green, blue and yellow subcomplexes are $K_{13}$, $K_{14}$, $K_{34}$ and $K_v$, respectively.}
    \label{im:dodecahedron}
\end{figure}

By Lemma \ref{lemma:QHS extension algorithm}, it is enough to find a row vector $v\in (\mathbb{Z}/2\mathbb{Z})^{12} \setminus \mathrm{Row}(\Lambda)$ such that the complexes $K_{\omega + v}$ are connected and have only trivial homology $1$--cycles for $\omega \in \{0, e^T_1 \Lambda, (e_1 + e_3)^T \Lambda, e^T_3 \Lambda\}$.

Recall that $\mathrm{Row}(\Lambda)=\{x^T \Lambda\mid x \in (\mathbb{Z}/2\mathbb{Z})^3\}$ and $\epsilon=(1, 1, 1)^T\Lambda$. Let $\omega= x^T \Lambda$ for some $x \in (\mathbb{Z}/2\mathbb{Z})^3$. Then for a face $F$ of $\mathcal{D}$ we have that $F^* \in K_{\omega}$ if and only if $x \cdot \lambda(F) = 1$. In particular, this means that the subcomplexes $K_{ij}=K_{(e_i+e_j)^T\Lambda}$ for $\{i,j\}\subset \{1,2,3\}$ are exactly the subcomplexes of $K$ with vertices coloured by $e_i$ and $e_j$, while the subcomplexes $K_{i4} = K_{(e_i)^T\Lambda}$ are the subcomplexes of $K$ with vertices coloured by $e_i$ and $e_1 + e_2 + e_3$.

Due to the constraint that $K_v$ be a rational homology point, the choice of vertices $(F_i)^*$ in $K$ such that $\mu(F_i)_4 = 1$ (or, equivalently, the choice of $v_i \neq 0$) should define one such subcomplex. By choosing $K_v$ as the simplex $\{3, 7, 9\}$ around which the three complexes $K_{13}$, $K_{14}$, $K_{34}$ are ``wrapped'', we have precisely that all four subcomplexes $K_{v+\omega}$ are rational homology points as shown in Figure~\ref{im:dodecahedron}.

Explicitly, the colouring $\mu$ with defining matrix
\begin{equation*}
\mathrm{M} = \begin{pmatrix}
1 & 0 & 0 & 0 & 0 & 1 & 1 & 1 & 1 & 1 & 0 & 0 \\
0 & 1 & 0 & 0 & 1 & 1 & 0 & 1 & 1 & 0 & 0 & 1 \\
0 & 0 & 1 & 1 & 0 & 1 & 0 & 1 & 1 & 0 & 1 & 0 \\
0 & 0 & 1 & 0 & 0 & 0 & 1 & 0 & 1 & 0 & 0 & 0 \\
\end{pmatrix}
\end{equation*}
is an extension of $\Lambda$ by Proposition \ref{prop:double-cover properties}, and $\mathcal{M}_\mu$ is a rational homology sphere by Equation~\eqref{eq:cohomology}. Hence, by Proposition~\ref{prop:arithmetic colourings that bound}, $M_\mu$ bounds geometrically. Finally, since $\Gamma(\mathcal{D})$ is arithmetic, it follows that $\mathcal{M}_\mu$ is also arithmetic.
\end{proof}

\begin{rem}\label{rem:final_remark}
\normalfont
A computer search among all possible extensions of the colouring $\Lambda$ from Theorem~\ref{thm:arithmetic-QHS-that-bounds} returned that there are, up to $DJ$--equivalence \cite[Definition 2.4]{FKS}, $7$ extensions which are rational homology $3$--spheres. However, the number of equivalence classes up to isometry might be smaller, given that the exact equivalence between isometry classes and colouring classes of compact hyperbolic $3$--polytopes holds only for small covers \cite[Theorem 3.13]{Vesnin}. 

\begin{table}[h!]
\vspace{0.25in}
\begin{tabular}{ |c|c|c| }
  \hline
  Colouring vector & $H_1(\mathcal{M}_\lambda,\mathbb{Z})$ & $\mathrm{Sym}_\lambda(\mathcal{P})$ \\ 
  \hline
  (1, 2, 4, 12, 10, 15, 9, 15, 7, 1, 4, 2) & $(\mathbb{Z}/2\mathbb{Z})^4 \times (\mathbb{Z}/4\mathbb{Z})^6$ & trivial \\
  (1, 2, 4, 12, 2, 15, 1, 7, 7, 9, 4, 2) & $(\mathbb{Z}/2\mathbb{Z})^8 \times (\mathbb{Z}/4\mathbb{Z})^4$ & $\mathbb{Z}/3\mathbb{Z}$ \\
  (1, 2, 4, 4, 10, 15, 1, 7, 7, 9, 4, 2) & $(\mathbb{Z}/2\mathbb{Z})^8 \times (\mathbb{Z}/4\mathbb{Z})^4$ & trivial \\
  (1, 2, 4, 12, 10, 15, 1, 15, 7, 9, 4, 2) & $(\mathbb{Z}/2\mathbb{Z})^4 \times (\mathbb{Z}/4\mathbb{Z})^6$ & $\mathbb{Z}/2\mathbb{Z}$ \\
  (1, 2, 4, 12, 10, 15, 1, 7, 15, 9, 4, 2) & $(\mathbb{Z}/2\mathbb{Z})^4 \times (\mathbb{Z}/4\mathbb{Z})^6$ & trivial \\
  \textcolor{blue}{(1, 2, 4, 4, 10, 7, 9, 15, 15, 9, 4, 2)} & \textcolor{blue}{$(\mathbb{Z}/2\mathbb{Z})^8 \times (\mathbb{Z}/4\mathbb{Z})^4$} & \textcolor{blue}{$\mathbb{Z}/3\mathbb{Z}$} \\
  (1, 2, 4, 12, 10, 7, 1, 15, 7, 9, 12, 2) & $(\mathbb{Z}/2\mathbb{Z})^4 \times (\mathbb{Z}/4\mathbb{Z})^6$ & $(\mathbb{Z}/2\mathbb{Z})^2$ \\
  \hline
\end{tabular}
\vspace{0.25in}
\caption{Extensions of $\lambda$ that produced rational homology spheres, as described in Remark~\ref{rem:final_remark}: highlighted in blue is the penultimate entry that corresponds to $\mu$.}
\label{im:tables}
\end{table}

In Table~\ref{im:tables}, we provide a representative of each colouring class, together with its first integral homology group and coloured symmetry group (cf. \cite[Section 2.3]{KS?} for more information on coloured symmetries). The volume of each $\mathcal{M}_\lambda$ in Table~\ref{im:tables} is $16 \cdot \mathrm{Vol}\,\mathcal{D} \approx 68.89936 \ldots$

Each colouring is represented by a colouring vector $v = (c_i)^{11}_{i=0} \in \mathbb{Z}^{12}$ that assigns the colour $c_i$ to the facet $F_i$ of the right-angled dodecahedron in Figure~\ref{im:dodecahedron}. The colour $c_i$ is given in the binary notation: if $c_i = (x, y, z, t) \in (\mathbb{Z}/2\mathbb{Z})^4$, then we use the map $c_i \mapsto x + 2y + 4z + 8t$. The colouring $\mu$ in the proof of Theorem~\ref{thm:arithmetic-QHS-that-bounds} is equivalent to the penultimate entry (highlighted in blue) of Table~\ref{im:tables}. For further details, we refer the reader to the auxiliary SageMath code available on GitHub \cite{github}. 
\end{rem}

\section{Magma computations}
\label{magma}

\subsection{Magma calculations for \S \ref{bound_compact}}
\label{magma_bound_compact}
Referring to the Magma \cite{Mag} code below, {\tt{g}} denotes the group $\Gamma$, and {\tt{K}} = {\tt{K1}} denotes the group $\Gamma_1$.

\begin{verbatim}
> g<x,y,z>:=Group<x,y,z|x^2,y^2,z^3,(y*z)^3,(z*x)^5,(x*y)^2>;
> P:=PSL(2,11);
> H := Homomorphisms(g, P: Limit := 2);
> print H;                                                                                            
[
Homomorphism of GrpFP: g into GrpPerm: P, Degree 12, Order 2^2 * 3 * 5 * 11 
induced by
        x |--> (1, 9)(2, 12)(3, 6)(4, 7)(5, 11)(8, 10)
        y |--> (1, 5)(2, 7)(3, 10)(4, 12)(6, 8)(9, 11)
        z |--> (1, 8, 2)(3, 4, 7)(5, 12, 11)(6, 9, 10),
Homomorphism of GrpFP: g into GrpPerm: P, Degree 12, Order 2^2 * 3 * 5 * 11 
induced by
        x |--> (1, 12)(2, 10)(3, 7)(4, 5)(6, 9)(8, 11)
        y |--> (1, 5)(2, 7)(3, 10)(4, 12)(6, 8)(9, 11)
        z |--> (1, 8, 2)(3, 4, 7)(5, 12, 11)(6, 9, 10)
]
> imgs:=[P!(1, 9)(2, 12)(3, 6)(4, 7)(5, 11)(8, 10),
P!(1, 5)(2, 7)(3, 10)(4, 12)(6, 8)(9, 11),
P!(1, 8,2)(3, 4, 7)(5, 12, 11)(6, 9, 10)];
> e := hom< g->P | imgs >; 
e(g) eq P;
true
> K:=Kernel(e);
> print AbelianQuotientInvariants(K);
[ 2, 2, 2, 2, 2, 2, 2, 22, 22, 22 ]
> K1:=Rewrite(g,K);
> IsPowerful := function (G)
function>    return DerivedGroup(G) subset Agemo (G, 1);
function> end function;
> H,A,B:=pQuotient(K1,11,2:Print:=1);

Lower exponent-11 central series for K1

Group: K1 to lower exponent-11 central class 1 has order 11^3

Group: K1 to lower exponent-11 central class 2 has order 11^6
> IsPowerful(H);
true
> l:=LowIndexSubgroups(K1,<2,2>);
> print #l;
1023
> M:=[x: x in l | not (0 in AbelianQuotientInvariants (x))];
> print #M;
363
> print AbelianQuotientInvariants(M[1]);
[ 2, 2, 2, 2, 2, 2, 2, 2, 4, 4, 44, 132, 132 ]
> M1:=Rewrite(K1,M[1]);
> HH,AA,BB:=pQuotient(M1,11,2:Print:=1);

Lower exponent-11 central series for M1

Group: M1 to lower exponent-11 central class 1 has order 11^3

Group: M1 to lower exponent-11 central class 2 has order 11^6
> IsPowerful(HH);
true
> M3:=Rewrite(K1,M[3]);                 
> HH,AA,BB:=pQuotient(M3,11,2:Print:=1);

Lower exponent-11 central series for M3

Group: M3 to lower exponent-11 central class 1 has order 11^3

Group: M3 to lower exponent-11 central class 2 has order 11^6
> IsPowerful(HH);                       
true
\end{verbatim}

\subsection{Magma calculations for \S \ref{ex:CD_BE}}
\label{magma_CD_BE}
Referring to the routine below, {\tt{g}} is the group $\Gamma_1$, and {\tt{K}} = {\tt{K1}} is the group $\Gamma_2$.

\medskip

\begin{verbatim}
g<a,b,c,d>:=Group<a,b,c,d|d^3,a*c*d*c*b^2*c*a*d^-1*c^-1,
a*c*b^2*d^-1*c^-1*a^-1*b^-1*d*b^-1,a*d^-1*a^-1*c^-1*b^-1*d*b*c,(b^2*d^-1)^3,
b*d^-1*b*c*a^-1*c*d*a^-1>;
> print AbelianQuotientInvariants(g);
[ 2, 6, 12 ]
H,A,B:=pQuotient(g,3,1:Print:=1);

Lower exponent-3 central series for g

Group: g to lower exponent-3 central class 1 has order 3^2
> K:=Kernel(A);
> print AbelianQuotientInvariants(K);
[ 6, 6, 36 ]
> K1:=Rewrite(g,K);
> l:=LowIndexSubgroups(K1,<2,2>); 
> print AbelianQuotientInvariants(l[1]);
[ 3, 3, 3, 3, 6, 6, 6, 6, 18 ]
> print AbelianQuotientInvariants(l[2]);
[ 3, 3, 3, 3, 18, 18, 18, 0 ]
> print AbelianQuotientInvariants(l[3]);
[ 3, 3, 3, 3, 36, 36, 0 ]
> print AbelianQuotientInvariants(l[4]);
[ 10, 30, 60, 180 ]
> print AbelianQuotientInvariants(l[5]);
[ 3, 3, 3, 3, 18, 18, 18, 0 ]
> print AbelianQuotientInvariants(l[6]);
[ 3, 3, 3, 3, 9, 9, 18, 0, 0 ]
> print AbelianQuotientInvariants(l[7]);
[ 10, 30, 60, 180 ]
> L4:=Rewrite(K1,l[4]);
> H,A,B:=pQuotient(L4,3,2:Print:=1);  

Lower exponent-3 central series for L4

Group: L4 to lower exponent-3 central class 1 has order 3^3

Group: L4 to lower exponent-3 central class 2 has order 3^6
> IsPowerful := function (G)
function>    return DerivedGroup(G) subset Agemo (G, 1);
function> end function;
> IsPowerful(H);
true
\end{verbatim}

\end{document}